\documentclass[12pt]{article}

\usepackage[margin=2cm,top=2cm,bottom=2cm]{geometry}
\usepackage[utf8]{inputenc}
\usepackage[T1]{fontenc}
\usepackage{xurl}
\usepackage{booktabs}
\usepackage{amsfonts}
\usepackage{microtype}

\usepackage{amsmath,amssymb,amsthm,mathtools,xcolor}
\usepackage{graphicx}
\usepackage{float}
\usepackage{algorithm,algpseudocode}
\usepackage[hidelinks,hypertexnames=false]{hyperref}
\allowdisplaybreaks
\usepackage{changepage}
\usepackage{indentfirst}

\newtheorem{theorem}{Theorem}[section]
\newtheorem{lemma}[theorem]{Lemma}
\newtheorem{proposition}[theorem]{Proposition}
\theoremstyle{remark}
\newtheorem{remark}[theorem]{Remark}
\theoremstyle{plain}
\numberwithin{equation}{section}

\newcommand{\R}{\mathbb R}
\newcommand{\Cl}{\partial_{\mathrm C}}
\newcommand{\norm}[1]{\lVert #1\rVert}
\newcommand{\SPD}{\mathbb S_{++}^{2}}
\newcommand{\proofstep}[2]{%
  \par\addvspace{\medskipamount}\noindent
  \underline{\textit{Step~#1: #2.}}\nobreak\hspace{0.6em}\ignorespaces
}
\title{BFGS can converge to non-Clarke-stationary points with positive probability}
\author{\large Lexiao Lai\thanks{\url{lai.lexiao@hku.hk}, Department of Mathematics, The University of Hong Kong, Hong Kong.}}
\date{}

\hypersetup{
  pdftitle={BFGS can converge to non-Clarke-stationary points with positive probability},
  pdfauthor={Lexiao Lai},
  pdfsubject={A semialgebraic counterexample to almost-sure stationarity of nonsmooth BFGS},
  pdfkeywords={nonsmooth optimization, BFGS, Clarke stationarity, semialgebraic functions, computer-assisted proof}
}

\begin{document}

\maketitle
\vspace*{-5mm}

\begin{center}
\textbf{Abstract}
\end{center}
\vspace*{-4mm}
\begin{adjustwidth}{0.2in}{0.2in}
We construct a coercive, locally Lipschitz,
semialgebraic function on $\R^2$ for which BFGS with Armijo--Wolfe
line search converges to a non-Clarke-stationary point from a nonempty
open set of initial points and inverse-Hessian estimates. This disproves the almost-sure stationarity assertion in the 2013 challenge
of Lewis and Overton \cite[Challenge~7.1(2)]{LewisOverton2013}. The construction and proof combine analytical arguments with computer-assisted verification of explicit inequalities.\footnote{Verification code and Lean formalization: \url{https://github.com/LexiaoLai/Nonsmooth-BFGS-Nonstationary}.}

\end{adjustwidth}

\section{Introduction and main result}
We consider the unconstrained minimization problem
$\min_{x\in\R^n} f(x)$, where $f:\R^n\to\R$ is locally Lipschitz
and may be nonsmooth and nonconvex. The
Broyden--Fletcher--Goldfarb--Shanno (BFGS) method approximates curvature
using gradient differences without evaluating Hessians
\cite{Broyden1970,Fletcher1970,Goldfarb1970,Shanno1970}.
Its numerical effectiveness and convergence properties in smooth
optimization have been studied extensively \cite{DennisMore1977}.
For twice continuously differentiable convex objectives with bounded
initial sublevel sets, BFGS function values converge to the minimum
under an Armijo--Wolfe line search \cite{Powell1976}.
Explicit global convergence rates are available for twice differentiable,
strongly convex objectives with Lipschitz gradients when the Armijo
parameter is below $1/2$ \cite[Theorem~4.1]{JinJiangMokhtari2024}.
For convex objectives with Lipschitz gradients, convergence rates have also been
established for the smallest gradient norm reached when the objective
is bounded below, and for the function-value gap when the initial
sublevel set is bounded \cite[Theorem~3.1]{DingYangZhou2026}.

BFGS has also shown empirical success on nonsmooth
objectives \cite{LewisOverton2013,CurtisMitchellOverton2017}.
Theoretical support for this behavior includes convergence results for the
Euclidean norm in two dimensions with
exact line search, and for $|x|$ with a particular inexact line search
using a zero Armijo parameter
\cite[Theorems~3.2 and~5.2]{LewisOverton2013}.
A doubling/bisection Armijo--Wolfe line search also terminates finitely
for bounded-below, locally Lipschitz, semialgebraic objectives along
directions with negative directional derivative
\cite[Theorem~4.7]{LewisOverton2013}.
These results and the numerical evidence motivate the broader question
posed in \cite[Challenge~7.1(2)]{LewisOverton2013}: for locally Lipschitz,
semialgebraic functions with bounded sublevel sets, is every cluster
point generated by BFGS with this line search Clarke stationary almost
surely under Gaussian initialization of the point and Wishart
initialization of the inverse-Hessian estimate?

Further results extend the analysis to broader nonsmooth classes.
For convex objectives satisfying additional structural assumptions,
BFGS function values converge to the minimum \cite{GuoLewis2018};
in particular, convergence holds for the Euclidean norm in every dimension
\cite[Corollary~4.2]{GuoLewis2018}.
For a class of unbounded-below convex functions with two affine pieces,
BFGS detects a direction of unbounded descent after finitely many
iterations, provided that the iterates remain differentiability points
\cite[Theorem~IV.1 and Remark~IV.1]{XieWaechter2017}.
For a class of piecewise-differentiable objectives, convergent sequences
initialized sufficiently near a minimizer have Clarke-stationary limits
under additional structural and asymptotic assumptions on the objective,
iterates, and quasi-Newton matrices
\cite[Theorem~3.1]{Gebken2026}.
These results do not settle the general nonsmooth convex case;
convergence for $(x_1,x_2)\mapsto x_1^2+|x_2|$ is identified as unresolved in
\cite[Section~1]{Gebken2026}.

Existing counterexamples show that BFGS and related methods can fail to
converge or converge to nonminimizers.
For smooth nonconvex objectives, examples exhibit cycling with Wolfe line
searches \cite{Dai2002} or failure with exact line searches, including
one with bounded sublevel sets \cite{Mascarenhas2004,Mascarenhas2014}.
For a prescribed initial point and inverse-Hessian estimate, BFGS with
Armijo--Wolfe line search generates bounded iterates with nonstationary
cluster points on a polynomial objective \cite[Sections~3.1--3.4]{Dai2013}.
The objective can be modified to have bounded sublevel sets without
changing the iterates.
For nonsmooth objectives, BFGS with an Armijo--Wolfe line search can
converge to a Clarke-stationary point admitting linear descent
\cite[Proposition~5.3]{LewisZhang2015}.
Scaled memoryless BFGS can converge to non-Clarke-stationary
points on an unbounded-below convex piecewise-affine objective
\cite[Corollary~3.12]{AslOverton2021}.
These examples do not establish failure of Clarke stationarity on an
open set of initializations for BFGS with the specified Armijo--Wolfe
line search on locally Lipschitz, semialgebraic objectives with bounded
sublevel sets.

In this work, we establish failure of BFGS that persists under perturbations of both the initial
point and the inverse-Hessian estimate. A single open set of initializations
gives convergence to the same non-Clarke-stationary point for every
line-search parameter pair in a fixed range. This set has positive
probability under Gaussian--Wishart initialization, disproving the
almost-sure stationarity assertion in
\cite[Challenge~7.1(2)]{LewisOverton2013}.

\subsection{Algorithm and main result}
We consider BFGS with the unit-trial doubling/bisection line search and strict
Armijo--Wolfe tests specified in
\cite[Eqs.~(4.2)--(4.3), Algorithm~4.6]{LewisOverton2013}; see Algorithm~\ref{alg:bfgs} below.
At iteration $k$, $x^k$ is the current point and $H^k\succ0$ is the
positive definite inverse-Hessian estimate. The constants
$c_1$ and $c_2$ are the Armijo and Wolfe parameters, respectively,
and $I$ denotes the identity matrix.

\begin{algorithm}[!ht]
\caption{BFGS with Armijo--Wolfe line search}
\label{alg:bfgs}
\small
\textbf{Input:} $f$, $x^{0}\in\R^n$, $H^{0}\succ0$, and $0<c_1<c_2<1$.
\begin{algorithmic}[1]
\State $g^{0}\gets\nabla f(x^{0})$.
\For{$k=0,1,2,\ldots$}
  \State Stop if $g^{k}=0$; set $p^{k}=-H^{k}g^{k}$.
  \State Set $\varphi^{k}(\alpha)=f(x^{k}+\alpha p^{k})$ and $(\ell,u,\alpha)=(0,+\infty,1)$.
  \Statex \hspace{\algorithmicindent}\begin{minipage}{0.91\linewidth}
  Accept a trial step size $\alpha$ when both strict tests hold:
  \begin{align}
   \varphi^{k}(\alpha)&<\varphi^{k}(0)+c_1\alpha (g^{k})^Tp^{k},
     &&\text{(Armijo)},\label{eq:armijo}\\
   (\varphi^{k})'(\alpha)&>c_2(g^{k})^Tp^{k},
     &&\text{(Wolfe)}.\label{eq:wolfe}
  \end{align}
  The Wolfe test fails if $(\varphi^{k})'(\alpha)$ does not exist.
  \end{minipage}
  \While{either \eqref{eq:armijo} or \eqref{eq:wolfe} fails}
    \State If \eqref{eq:armijo} fails, set $u=\alpha$; otherwise set $\ell=\alpha$.
    \State Set $\alpha=(\ell+u)/2$ if $u<+\infty$, and $\alpha=2\ell$ otherwise.
  \EndWhile
  \State $\alpha^{k}\gets\alpha$, $s^{k}\gets\alpha^{k}p^{k}$, $x^{k+1}\gets x^{k}+s^{k}$.
  \State $g^{k+1}\gets\nabla f(x^{k+1})$, $y^{k}\gets g^{k+1}-g^{k}$.
  \State $H^{k+1}\gets\mathcal B(H^{k},s^{k},y^{k})$, where
  \Statex \hspace{\algorithmicindent}\begin{minipage}{0.91\linewidth}
  \begin{equation}\label{eq:update}
   \mathcal B(H,s,y)=\left(I-\frac{sy^T}{s^Ty}\right)
   H\left(I-\frac{ys^T}{s^Ty}\right)+\frac{ss^T}{s^Ty}.
  \end{equation}
  \end{minipage}
\EndFor
\end{algorithmic}
\end{algorithm}

Whenever $f$ is differentiable at the accepted point $x^{k+1}$,
the Wolfe condition \eqref{eq:wolfe} gives
$(s^{k})^Ty^{k}>\alpha^{k}(c_2-1)(g^{k})^Tp^{k}>0$.
The update therefore preserves positive definiteness of inverse-Hessian estimates, so each subsequent
nonzero gradient generates a descent direction. To state the main result, recall that for locally Lipschitz $f$,
the Clarke subdifferential $\Cl f(x)$ is
the closed convex hull of limiting gradients at differentiability
points approaching $x$ \cite{Clarke1975,Clarke1990}.
A point $x$ is Clarke stationary when $0\in\Cl f(x)$.
A function is semialgebraic if its graph is a finite Boolean combination
of sets defined by polynomial equalities and inequalities
\cite[Sections~2.1--2.2]{BochnakCosteRoy1998}.
We use $\norm{\cdot}$ for the Euclidean norm and call $f$ coercive if
$f(x)\to+\infty$ as $\norm{x}\to\infty$. Let $\SPD\subset \mathbb R^{2\times2}$ be the set of
symmetric positive definite $2\times2$ matrices. The following is the main result of this paper.

\begin{theorem}\label{thm:main}
There exist a coercive, locally Lipschitz, semialgebraic function
$f:\R^2\to\R$ with $0\notin\Cl f(0)$ and a nonempty open set
$\mathcal U\subset\R^2\times\SPD$ such that,
for every $(x^{0},H^{0})\in\mathcal U$ and every parameter pair $(c_1,c_2)$
satisfying
\begin{equation}\label{eq:parameter-range}
 0<c_1\le0.03<0.6\le c_2<1,
\end{equation}
Algorithm~\ref{alg:bfgs} generates a well-defined infinite sequence
with $x^{k}\to0$.
\end{theorem}

Figure~\ref{fig:main-two-panel} shows the objective in Theorem~\ref{thm:main}
and a sequence of BFGS iterates converging to the non-Clarke-stationary origin.
Figure~\ref{fig:random-initializations} plots the function values and
iterate norms for 100 BFGS runs initialized by randomly perturbing
the initialization used in
Figure~\ref{fig:main-two-panel}.

\begin{figure}[!ht]
\centering
\includegraphics[width=\linewidth]{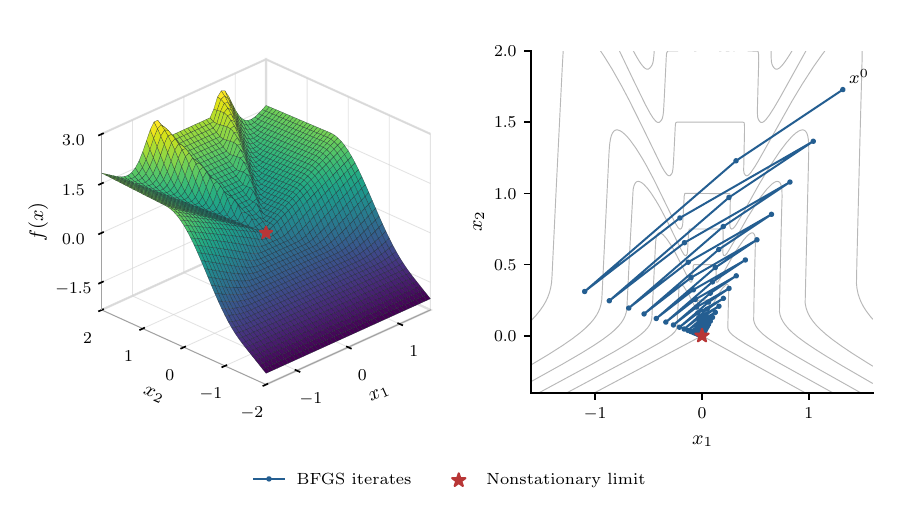}
\caption{BFGS convergence to a non-Clarke-stationary origin.
Left: the objective landscape near the cycle.
Right: the first 120 BFGS iterations with
$(c_1,c_2)=(0.01,0.9)$.}
\label{fig:main-two-panel}
\end{figure}

\begin{figure}[!ht]
\centering
\includegraphics[width=\linewidth]{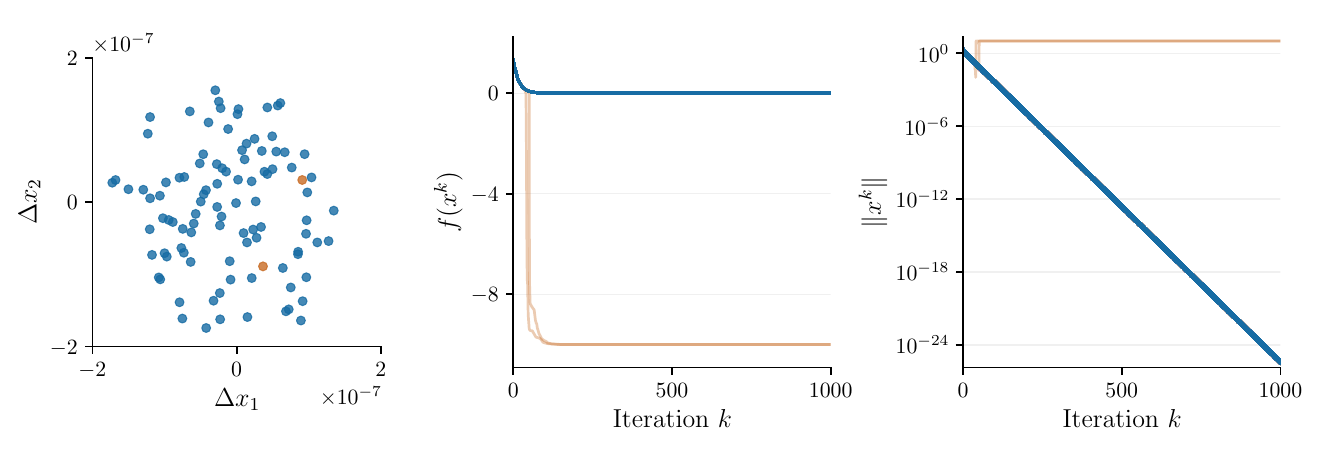}
\caption{BFGS from 100 random initializations for the objective in
Figure~\ref{fig:main-two-panel}, with $(c_1,c_2)=(0.01,0.9)$ and 1000
iterations per run. The initial states are sampled uniformly from the
five-dimensional Euclidean ball of radius $\varepsilon=2\times10^{-7}$ centered
at the constructed state $(x^0,H^0)$, in coordinates
$(x_1,x_2,H_{11},H_{12},H_{22})$.
The left panel shows the initial position displacements
$(\Delta x_1,\Delta x_2)$, while the middle and right panels show the
objective values and iterate norms, respectively. Runs converging to the
origin are shown in blue, and the remaining runs in orange.}
\label{fig:random-initializations}
\end{figure}

\begin{remark}[Positive probability of failure]\label{rem:probability}
If $x^{0}$ is a random Gaussian vector with a positive definite covariance matrix and, independently,
$H^{0}=X^TX$ with independent standard Gaussian entries in the $2\times2$
matrix $X$, every nonempty open subset of $\R^2\times\SPD$ has positive
probability. The same event $\mathcal U$ gives failure for every parameter
pair in \eqref{eq:parameter-range}, disproving the
almost-sure stationarity assertion in
\cite[Challenge~7.1(2)]{LewisOverton2013}.
\end{remark}

The proof of Theorem~\ref{thm:main} in Section~\ref{sec:proof} combines a homogeneous shrinking cycle with stability under perturbations of the initial state.
In Section~\ref{sec:cycle}, a scalar polynomial equation and a linear involution determine a four-step pattern. After four updates, both the iterate and the inverse-Hessian estimate are scaled by $64/81$. Once realized by a positively homogeneous function, this pattern repeats at successively smaller scales, so the iterates converge to the origin.
Section~\ref{sec:realization} constructs this function by interpolation.
It matches the prescribed values and gradients while keeping its Clarke
subgradients uniformly separated from zero. The interpolation also allows
us to prescribe compatible Hessians for the subsequent stability argument.
Section~\ref{sec:stability} verifies that, with the
chosen Hessians and after dividing out the geometric decay, the four-step return map
contracts directions transverse to the ray of scaled cycle states, and
establishes an invariant neighborhood. Finally, Section~\ref{sec:completion}
makes the objective coercive while preserving local Lipschitz continuity, then proves convergence from an open set of initializations.

The proof relies on computer-assisted certification of the finite inequalities
in Appendix~\ref{app:certificates}, using rigorous interval arithmetic
\cite{MooreKearfottCloud2009,Rump2010}.
We also use symbolic computation in Python
to confirm the cycle identities proved analytically in Section~\ref{sec:cycle}.
The verification details, executed Jupyter notebook, and Lean
formalization of the main result are available in the accompanying GitHub repository.

\section{Construction and proof}\label{sec:proof}
Recall that a function $\phi:\R^n\to\R$ is positively homogeneous of
degree one if $\phi(tx)=t\phi(x)$ for every $x\in\R^n$ and $t\ge0$.
We use ``homogeneous'' in this sense throughout. We first record the consequences of homogeneity used in the construction
and in the analysis of BFGS.

\begin{proposition}[Homogeneity identities]\label{prop:homogeneity}
Let $\phi:\R^n\to\R$ be positively homogeneous and $C^2$ near $x\in\R^n$. Then, for every $\tau>0$, $\phi$ is $C^2$ near $\tau x$, and
\begin{equation}\label{eq:homogeneity-identities}
 x^T\nabla\phi(x)=\phi(x),\qquad
 \nabla\phi(\tau x)=\nabla\phi(x),\qquad
 \nabla^2\phi(\tau x)=\tau^{-1}\nabla^2\phi(x).
\end{equation}
\end{proposition}

\begin{proof}
Let $U$ be an open neighborhood of $x$ on which $\phi$ is $C^2$.
For fixed $\tau>0$, homogeneity gives $\phi(z)=\tau\phi(z/\tau)$
on $\tau U$. The right-hand side is $C^2$ there, so $\phi$ is $C^2$
in a neighborhood of $\tau x$.
Differentiating $\phi(tx)=t\phi(x)$ with respect to $t$ at $t=1$
gives the first identity in \eqref{eq:homogeneity-identities}.
Differentiating $\phi(\tau z)=\tau\phi(z)$ once and twice
with respect to $z$ at $x$ gives the derivative scaling identities.
\end{proof}

We next record how BFGS updates transform under scaling and invertible
linear changes of coordinates. For homogeneous objectives, scaling the
point and inverse-Hessian estimate together also preserves the
line-search decisions.

\begin{lemma}[BFGS scaling and covariance]\label{lem:bfgs-scaling}
For $H\in\R^{n\times n}$, invertible $A\in\R^{n\times n}$,
$s,y\in\R^n$ with $s^Ty\ne0$, and $\tau>0$, the update in
\eqref{eq:update} satisfies
\begin{equation}\label{eq:bfgs-transformations}
 \mathcal B(\tau AHA^T,\tau As,A^{-T}y)
 =\tau A\mathcal B(H,s,y)A^T.
\end{equation}

Let $\phi$ and $x$ satisfy the hypotheses of
Proposition~\ref{prop:homogeneity}, and let
$H\succ0$, $\nabla\phi(x)\ne0$ and $\tau>0$.
Applied to $\phi$ with the same line-search parameters,
Algorithm~\ref{alg:bfgs} tests the same step sizes and makes the same
acceptance decisions from $(x,H)$ and $(\tau x,\tau H)$.
If the first run accepts a step and produces $(x^{1},H^{1})$,
and $\phi$ is $C^2$ near $x^{1}$, then the second accepts the same
step size and produces $(\tau x^{1},\tau H^{1})$.
\end{lemma}

\begin{proof}
The transformed scalar product is
$(\tau As)^T(A^{-T}y)=\tau s^Ty$, and
\[
 \begin{aligned}
 I-\frac{(\tau As)(A^{-T}y)^T}{\tau s^Ty}
 &=A\left(I-\frac{sy^T}{s^Ty}\right)A^{-1},\\
 \frac{(\tau As)(\tau As)^T}{\tau s^Ty}
 &=\tau A\frac{ss^T}{s^Ty}A^T.
 \end{aligned}
\]
Substitution into \eqref{eq:update}, together with the transpose of the
first identity, proves \eqref{eq:bfgs-transformations}.

Set $p=-H\nabla\phi(x)$. By the gradient scaling identity in
Proposition~\ref{prop:homogeneity}, the search direction at
$(\tau x,\tau H)$ is $\tau p$.
For the line functions $\varphi(\alpha)=\phi(x+\alpha p)$ and
$\widehat\varphi(\alpha)=\phi(\tau x+\alpha\tau p)$, homogeneity gives
$\widehat\varphi=\tau\varphi$. Thus a line derivative exists at a trial
in one search exactly when it exists in the other, and then
$\widehat\varphi'=\tau\varphi'$.
Both sides of the Armijo and Wolfe tests
\eqref{eq:armijo}--\eqref{eq:wolfe} are multiplied by $\tau>0$,
so the decisions and subsequent trial step sizes coincide.
If $\phi$ is $C^2$ near the accepted iterate $x^{1}=x+\alpha p$,
Proposition~\ref{prop:homogeneity} applies at both $x$ and $x^{1}$.
Thus the displacement is multiplied by $\tau$ and the gradient
difference is unchanged.
The Wolfe condition gives positive curvature, so the update is defined;
\eqref{eq:bfgs-transformations} with $A=I$ proves the assertion about
the resulting states.
\end{proof}

\subsection{An exact four-step cycle}\label{sec:cycle}
We construct the first two updates and obtain the next two by the scaling
and covariance identity in Lemma~\ref{lem:bfgs-scaling}.
Set $\rho=8/9$ and $r=\rho^2=64/81$.
We will choose a linear involution $J\in\R^{2\times2}$, meaning $J^2=I$,
so that the data after two updates are related to the initial data by
\begin{equation}\label{eq:cycle-state-transformation}
 (x,H,g)\longmapsto(\rho J^Tx,\rho J^THJ,Jg).
\end{equation}
Lemma~\ref{lem:bfgs-scaling}, with $A=J^T$ and $\tau=\rho$,
then transforms updates $0\to1\to2$ into updates $2\to3\to4$.
Applying this transformation twice gives $(rx,rH,g)$, so the four updates
return the point and inverse-Hessian estimate to their initial values
multiplied by $r$.

For now, let $b>1/2$ be a free parameter. Let $e_1=(1,0)^T$ and
$e_2=(0,1)^T$ be the standard basis vectors, and choose the first two
gradients as $g^0=e_1$ and $g^1=e_2$.
Choose the first column of $H^0$ as $(1,1/2)^T$. The update
\eqref{eq:update} with displacement $s^0  = p^0 = -H^0e_1$ then gives
$(H^1)_{22}=4(H^0)_{22}-1/2$.
We use $b=(H^1)_{22}$ as the parameter to simplify the second update,
so $(H^0)_{22}=b/4+1/8$. We therefore set
\begin{equation}\label{eq:initial-metrics}
 H^0=\begin{pmatrix}1&1/2\\1/2&b/4+1/8\end{pmatrix},\qquad
 H^1=\begin{pmatrix}b+3/2&b+1/2\\b+1/2&b\end{pmatrix}.
\end{equation}

We seek a second update with $g^2=Je_1$ and
$H^2=\rho J^TH^0J$, where
\begin{equation}\label{eq:involution-form}
 J=\begin{pmatrix}p&t\\q&-p\end{pmatrix}
\end{equation}
has entries to be determined and must satisfy $J^2=I$.
The next lemma verifies a proposed two-dimensional BFGS update by
checking its secant equation and determinant.

\begin{lemma}[A two-dimensional update test]\label{lem:secant-det}
Let $H\in\SPD$ and $s,y\in\R^2$ satisfy $s^Ty>0$.
A symmetric matrix $K\in\R^{2\times2}$ equals $\mathcal B(H,s,y)$
if and only if
\begin{equation}\label{eq:update-test}
 Ky=s,\qquad \det K=\det H\,\frac{s^TH^{-1}s}{s^Ty}.
\end{equation}
\end{lemma}

\begin{proof}
The update \eqref{eq:update} is symmetric and satisfies
$\mathcal B(H,s,y)y=s$. Choose an orthonormal basis with
$s=(\sigma,0)^T$, where $\sigma=\norm s>0$, and write $y=(u,v)^T$ and
$H=\left(\begin{smallmatrix}a&h\\h&d\end{smallmatrix}\right)$.
Since $\sigma u=s^Ty>0$, we have $u>0$, and substitution into
\eqref{eq:update} gives
\[
 \mathcal B(H,s,y)=
 \begin{pmatrix}dv^2/u^2+\sigma/u&-dv/u\\-dv/u&d\end{pmatrix}.
\]
Its determinant is $d\sigma/u$. Since
$s^TH^{-1}s=d\sigma^2/\det H$, this is the determinant in
\eqref{eq:update-test}, proving necessity.

Conversely, in an orthonormal basis whose first vector is $y/\norm y$,
the equation $Ky=s$ fixes the first column and, by symmetry, the first
row of $K$. Its first diagonal entry is $s^Ty/\norm y^2>0$, so its
determinant fixes the remaining entry. Hence the two conditions in
\eqref{eq:update-test} determine $K$ uniquely, and the matrix is
$\mathcal B(H,s,y)$.
\end{proof}

We first derive the second-update conditions assuming $J^2=I$, and
then choose $b$ so that the resulting matrix satisfies this identity.
Under this assumption, \eqref{eq:involution-form} gives
$\det J=-1$. By \eqref{eq:initial-metrics},
$\det H^2/\det H^1=\rho^2/2=32/81$.
Set $s=-H^1g^1$ and $y=g^2-g^1$. Then
$s^T(H^1)^{-1}s=(g^1)^TH^1g^1=b$, so the determinant condition
\eqref{eq:update-test}, with $H=H^1$ and $K=H^2$, becomes
\begin{equation}\label{eq:second-determinant}
 \frac{32}{81}=\frac{\det H^2}{\det H^1}
 =\frac{s^T(H^1)^{-1}s}{s^Ty}
 =\frac{b}{(g^2-g^1)^T(-H^1g^1)}.
\end{equation}
Rearranging gives $(g^2-g^1)^T(-H^1g^1)=81b/32$.
For the secant equation $H^2(g^2-g^1)=-H^1g^1$, substituting
$H^2=\rho J^TH^0J$ and $g^2=Je_1$, multiplying by $J^T$ and
using $J^2=I$ gives the equivalent equation
\begin{equation}\label{eq:second-secant}
 \rho H^0(e_1-Je_2)=-J^TH^1e_2.
\end{equation}
Expanding the curvature identity obtained from \eqref{eq:second-determinant}
and the secant equation \eqref{eq:second-secant}, using \eqref{eq:initial-metrics},
\eqref{eq:involution-form} and $\rho=8/9$, gives
\begin{equation}\label{eq:involution-conditions}
 \begin{aligned}
 -(b+\tfrac12)p-b(q-1)&=\frac{81b}{32},\\
 \frac89(1-t+\tfrac12p)&=-(b+\tfrac12)p-bq,\\
 \frac89\bigl(\tfrac12-\tfrac12t+(\tfrac b4+\tfrac18)p\bigr)
 &=-(b+\tfrac12)t+bp.
 \end{aligned}
\end{equation}
 For $b>1/2$, this linear system has
a unique solution:
\begin{equation}\label{eq:involution-data}
 \begin{gathered}
 p=-\frac{9(882b^2-463b-256)}{640(2b-1)},\qquad
 t=-\frac{3087b^2-2233b-256}{320(2b-1)},\\
 q=\frac{15876b^3-4316b^2-6815b-2304}{1280b(2b-1)}.
 \end{gathered}
\end{equation}
Since $J^2=(p^2+qt)I$, it remains to impose $p^2+qt=1$.
Substitution reduces this requirement to
\begin{equation}\label{eq:involution}
 p^2+qt-1=\frac{9P(b)}{409600b(2b-1)} = 0,
\end{equation}
where
\begin{equation}\label{eq:quartic}
 P(b)=777924b^4-576632b^3-587665b^2+352256b+65536.
\end{equation}
Thus, $J^2=I$ if and only if $P(b)=0$ for $b>1/2$.
The following lemma selects an admissible value of the parameter $b$.

\begin{lemma}[Admissible root]\label{lem:root-certificate}
The polynomial $P$ has a unique root $b$ satisfying
\[
 0.91781039164<b<0.91781039165.
\]
\end{lemma}

\begin{proof}
Set $a=0.91781039164$ and $c=0.91781039165$.
Exact rational evaluation gives $P(a)<0<P(c)$. For $z\in[a,c]$,
\[
 P'(z)\ge3111696a^3-1729896c^2-1175330c+352256>200000.
\]
The intermediate value theorem gives a root in $(a,c)$,
and the positive derivative makes it unique in $(a,c)$.
\end{proof}

For the rest of the construction, fix the root $b$ supplied by
Lemma~\ref{lem:root-certificate}. Equation~\eqref{eq:involution} now gives
$J^2=I$ and $\det J=-1$.
Complete the gradient sequence by setting
\begin{equation}\label{eq:cycle-gradients}
 (g^0,g^1,g^2,g^3,g^4)=(e_1,e_2,Je_1,Je_2,e_1).
\end{equation}
With $H^0,H^1$ from \eqref{eq:initial-metrics}, define the remaining
matrices by
\begin{equation}\label{eq:cycle-metrics}
 H^2=\rho J^TH^0J,\qquad
 H^3=\rho J^TH^1J,\qquad H^4=rH^0.
\end{equation}
Set $v_0=H^0g^0+H^1g^1=(b+3/2,b+1/2)^T$, using
\eqref{eq:initial-metrics} and \eqref{eq:cycle-gradients}.
The prescribed displacements $-H^0g^0$ and $-H^1g^1$ sum to $-v_0$, so requiring
$x^2=\rho J^Tx^0$ gives $(I-\rho J^T)x^0=v_0$.
Since $J^2=I$ and $r=\rho^2<1$, we have
$(I-\rho J^T)^{-1}=(I+\rho J^T)/(1-r)$.
We therefore define the positions by
\begin{equation}\label{eq:cycle-positions}
 \begin{gathered}
 x^0=\frac{(I+\rho J^T)v_0}{1-r},\qquad
 x^1=x^0-(1,1/2)^T,\\
 x^2=\rho J^Tx^0,\qquad
 x^3=\rho J^Tx^1,\qquad x^4=rx^0.
 \end{gathered}
\end{equation}
To satisfy the homogeneity identity in Proposition~\ref{prop:homogeneity},
prescribe $f^i=(g^i)^Tx^i$ for $0\le i\le4$.
Equations~\eqref{eq:cycle-gradients} and \eqref{eq:cycle-positions} give
$g^4=g^0$ and $f^4=rf^0$, as required by
\eqref{eq:homogeneity-identities} at $x^4=rx^0$.
Define
\begin{equation}\label{eq:cycle-recursion}
 s^i=-H^ig^i,\qquad y^i=g^{i+1}-g^i\quad(0\le i<4).
\end{equation}

The next lemma verifies the cycle, positive curvature, and position bounds.

\begin{lemma}[Exact cycle]\label{lem:cycle}
For the data just constructed,
$H^i\succ0$ and $\norm{x^i}^2<8$ for $0\le i\le4$, and
\begin{equation}\label{eq:cycle}
 \begin{gathered}
 x^{i+1}=x^i+s^i,\qquad H^{i+1}=\mathcal B(H^i,s^i,y^i)
 \quad(0\le i<4),\\
 H^4=rH^0,\qquad x^4=rx^0.
 \end{gathered}
\end{equation}
The curvature products are positive and satisfy
\begin{equation}\label{eq:curvatures}
 \bigl((s^i)^Ty^i\bigr)_{i=0}^3
 =\left(\frac12,\frac{81b}{32},\frac49,\frac{9b}{4}\right).
\end{equation}
\end{lemma}

\begin{proof}
\proofstep{1}{Positive definiteness and position bounds}
Lemma~\ref{lem:root-certificate} gives $P(b)=0$ and, in particular,
$b\in(0.9178,0.9179)$.
Thus \eqref{eq:involution} gives $J^2=I$.
For $H^0$ and $H^1$ in \eqref{eq:initial-metrics}, the determinants are
$(2b-1)/8$ and $(2b-1)/4$, respectively. Since $b>1/2$, both
determinants and both leading diagonal entries are positive.
The remaining matrices in \eqref{eq:cycle-metrics} are positive multiples of
congruences of these matrices, so all five are positive definite.

Substituting the bounds on $b$ into \eqref{eq:involution-data} and
\eqref{eq:cycle-positions} gives $|x_j^i|<2$ for $j=1,2$ and
$0\le i<4$. The definition $x^4=rx^0$, with $0<r<1$, gives the same
coordinate bounds for $x^4$. Hence $\norm{x^i}^2<8$ for
$0\le i\le4$.

\proofstep{2}{The first update}
By \eqref{eq:cycle-recursion},
$(s^i)^T(H^i)^{-1}s^i=(g^i)^TH^ig^i$ for $0\le i<4$.
For the first update, \eqref{eq:cycle-gradients} gives $g^0=e_1$ and
$g^1=e_2$, so
$s^0=-H^0e_1=(-1,-1/2)^T$ and $y^0=e_2-e_1=(-1,1)^T$.
Hence
\[
 H^1y^0
 =\begin{pmatrix}b+3/2&b+1/2\\b+1/2&b\end{pmatrix}
  \begin{pmatrix}-1\\1\end{pmatrix}
 =\begin{pmatrix}-(b+3/2)+(b+1/2)\\-(b+1/2)+b\end{pmatrix}
 =\begin{pmatrix}-1\\-1/2\end{pmatrix}=s^0.
\]
Moreover, $(s^0)^Ty^0=(-1)(-1)+(-1/2)(1)=1/2$ and
$(g^0)^TH^0g^0=e_1^TH^0e_1=(H^0)_{11}=1$.
Together with $\det H^1/\det H^0=2$, these identities verify
\eqref{eq:update-test}, so Lemma~\ref{lem:secant-det} proves the first update.

\proofstep{3}{The second update}
The parameters in \eqref{eq:involution-data} satisfy
\eqref{eq:involution-conditions} by construction.
By \eqref{eq:cycle-gradients} and \eqref{eq:cycle-recursion}, the vectors
$s,y$ used in \eqref{eq:second-determinant} are $s^1,y^1$.
The first relation in \eqref{eq:involution-conditions} therefore gives
$(s^1)^Ty^1=81b/32>0$.
The last two relations give \eqref{eq:second-secant}.
Using $J^2=I$ and \eqref{eq:cycle-metrics}, and multiplying
\eqref{eq:second-secant} by $J^T$, yields
\[
 H^2y^1=\rho J^TH^0(e_1-Je_2)=-H^1e_2=s^1.
\]
The determinant condition in \eqref{eq:second-determinant} is also
satisfied. Hence Lemma~\ref{lem:secant-det} gives
$H^2=\mathcal B(H^1,s^1,y^1)$.

\proofstep{4}{The third and fourth updates}
We express their data in terms of those of the first two updates
using the transformation in \eqref{eq:cycle-state-transformation}.
By \eqref{eq:cycle-gradients} and $J^2=I$, the gradients satisfy
$g^{i+2}=Jg^i$ for $i=0,1,2$.
The definitions in \eqref{eq:cycle-metrics} give
$H^{i+2}=\rho J^TH^iJ$ for $i=0,1$; the same identity holds for $i=2$ because
\[
 \rho J^TH^2J=\rho^2(J^T)^2H^0J^2=rH^0=H^4.
\]
Thus, for $i=0,1$, the definitions in \eqref{eq:cycle-recursion} give
\begin{equation}\label{eq:cycle-step-transformations}
 \begin{aligned}
 s^{i+2}&=-H^{i+2}g^{i+2}
 =-\rho J^TH^iJ^2g^i=\rho J^Ts^i,\\
 y^{i+2}&=g^{i+3}-g^{i+2}
 =J(g^{i+1}-g^i)=Jy^i.
 \end{aligned}
\end{equation}
In particular,
\[
 (s^{i+2})^Ty^{i+2}
 =\rho(s^i)^TJ^2y^i=\rho(s^i)^Ty^i>0\qquad(i=0,1).
\]
Apply Lemma~\ref{lem:bfgs-scaling} with $A=J^T$ and $\tau=\rho$,
so that $A^{-T}=J$. Using the first two updates proved in Steps~2--3,
\eqref{eq:bfgs-transformations} gives
\[
 \mathcal B(H^{i+2},s^{i+2},y^{i+2})=\rho J^T\mathcal B(H^i,s^i,y^i)J=\rho J^TH^{i+1}J=H^{i+3}\qquad(i=0,1).
\]
Taking $i=0$ and $i=1$ proves the third and fourth updates, respectively.
The curvature identity also gives $(s^2)^Ty^2=\rho/2=4/9$ and
$(s^3)^Ty^3=\rho(81b/32)=9b/4$.

\proofstep{5}{Position identities}
Since $s^0+s^1=-v_0$ and
$(I-\rho J^T)(I+\rho J^T)=(1-r)I$, the definition of $x^0$ in
\eqref{eq:cycle-positions} gives
$x^0+s^0+s^1=\rho J^Tx^0=x^2$.
The other position identities follow from \eqref{eq:cycle-positions}
and the step transformations in \eqref{eq:cycle-step-transformations}.
\end{proof}

\subsection{Realizing the objective and the line search}\label{sec:realization}
We now construct a homogeneous function $\widetilde f:\mathbb R^2\to\mathbb R$
that realizes the cycle and the line-search decisions. We use the cycle
positions $x^i$, matrices $H^i$, gradients $g^i$, values $f^i$, steps $s^i$,
gradient differences $y^i$, and contraction factor $r$ from
Section~\ref{sec:cycle}. To exclude Clarke stationarity at the origin,
we will keep the directional derivative along one fixed vector uniformly
positive. Consider
\begin{equation}\label{eq:profile-coordinates}
 M=\begin{pmatrix}1&0.05\\0&1\end{pmatrix},\qquad
 v=Me_2=(0.05,1)^T.
\end{equation}
We will show later in Proposition~\ref{prop:interpolation-certificate}
that $(g^i)^Tv>0$ for $0\le i<4$.
We construct $F:\R^2\to\R$ so that $\widetilde f(x)=F(M^{-1}x)$ has
gradients near the origin whose inner products with $v$ are uniformly positive.
Since $Me_2=v$, the chain rule gives
$v^T\nabla\widetilde f(Mz)=\partial_2F(z)$ wherever $F$ is differentiable.
Thus a uniform positive lower bound on $\partial_2F$ gives the desired
directional bound along $v$.

The construction of $F$ reduces to interpolating one univariate function.
We seek a $C^2$ semialgebraic function $h:\R\to\R$ whose first
derivative has a positive lower bound and which satisfies
$h(t)=t$ for sufficiently large $|t|$.
Use this function to define $F$ by
\begin{equation}\label{eq:homogeneous}
 F(z)=\begin{cases}
 |z_1|h(z_2/|z_1|),&z_1\ne0,\\
 z_2,&z_1=0,
 \end{cases}
\end{equation}
Positive scaling leaves the ratio unchanged and scales the prefactor,
so $F$ and $\widetilde f$ are positively homogeneous.
For $z_1\ne0$, partial differentiation gives
$\partial_2F(z)=h'(z_2/|z_1|)$.
The identity tails of $h$ make $F(z)=z_2$ near each nonzero point of the axis
$z_1=0$.

To construct $h$, we use the following
consequence of shape-preserving Hermite
interpolation~\cite[Sections~4.1--4.2]{Manni2001}. It provides an interpolant
that matches the prescribed values and first two derivatives while
keeping its first derivative positive.

\begin{lemma}[Monotone interpolation]\label{lem:interpolation}
Let $t_1<\cdots<t_N$ be interpolation points, with $N\ge1$.
For $1\le j\le N$, let $\eta_j,p_j,P_j\in\R$ be prescribed values of the
function and its first two derivatives at $t_j$. Suppose $0<m<1$ and
\[
 p_j>m\quad(1\le j\le N),\qquad
 \frac{\eta_{j+1}-\eta_j}{t_{j+1}-t_j}>m\quad(1\le j<N).
\]
Then there is a $C^2$ semialgebraic function $h:\R\to\R$ with
$h'(t)\ge m$ for every $t\in\R$ and
\[
 h(t_j)=\eta_j,\qquad h'(t_j)=p_j,\qquad h''(t_j)=P_j
 \quad(1\le j\le N).
\]
Moreover, the affine tails can be chosen as
\[
 h(t)=t+\eta_1-t_1\quad(t\le t_1-1),\qquad
 h(t)=t+\eta_N-t_N\quad(t\ge t_N+1).
\]
\end{lemma}

\begin{proof}
Adjoin the points
$t_0=t_1-1$ and $t_{N+1}=t_N+1$, with
\[
 \eta_0=\eta_1-1,\qquad \eta_{N+1}=\eta_N+1,\qquad
 p_0=p_{N+1}=1,\qquad P_0=P_{N+1}=0.
\]
The two new secant slopes are one, so the extended data satisfy the
same strict inequalities. Consider the shifted Hermite data
\[
 (t_j,\eta_j-mt_j,p_j-m,P_j)\qquad(0\le j\le N+1).
\]
Their consecutive value increments and prescribed first derivatives
are positive. They therefore satisfy the monotonicity
conditions in~\cite[Definition~2.1]{Manni2001}: the strict positivity of
$p_j-m$ makes the endpoint conditions hold for sufficiently small
positive parameters, regardless of $P_j$.
The parametric construction~\cite[Section~4.1 and Section~4.2, equations~(4.7)--(4.9)]{Manni2001}
then gives a $C^2$ interpolant $g$ on $[t_0,t_{N+1}]$ with $g'\ge0$.
Its graph consists of finitely many polynomially parametrized arcs,
so it is semialgebraic by the Tarski--Seidenberg
theorem~\cite[Theorem~2.2.1]{BochnakCosteRoy1998}.

Set $h(t)=g(t)+mt$ on $[t_0,t_{N+1}]$ and extend it by the two
affine formulas in the statement. The function $g$ matches the shifted
data, so $h$ matches the original data and satisfies $h'\ge m$ on this
interval. The affine extensions match its endpoint values, first
derivatives one and second derivatives zero. Hence $h$ is $C^2$ and
semialgebraic on $\R$, with $h'\ge m$ everywhere.
\end{proof}

To prescribe the interpolation data, express the cycle positions in
the new coordinates as
\begin{equation}\label{eq:transformed-cycle-positions}
 z^i=M^{-1}x^i\quad(0\le i<4).
\end{equation}
The next proposition shows that the slope parameters $z_2^i/|z_1^i|$
are well defined and gives the order in which their data are interpolated.

\begin{proposition}[Slope parameters]\label{prop:geometry-certificate}
The points in \eqref{eq:transformed-cycle-positions} satisfy
$z_1^i\ne0$ for $0\le i<4$.
Writing $t^i=z_2^i/|z_1^i|$, we have
$-1<t^2<t^0<t^3<t^1<6$.
\end{proposition}

\begin{proof}
Substituting the refined bounds from Lemma~\ref{lem:root-certificate}
into \eqref{eq:involution-data}, \eqref{eq:cycle-positions},
\eqref{eq:profile-coordinates} and \eqref{eq:transformed-cycle-positions}
bounds $z_1^i$ away from zero.
Dividing $z_2^i$ by $|z_1^i|$ then gives the following enclosures:
\begin{center}
\begin{tabular}{ccc}
\toprule
$i$ & $z_1^i$ & $t^i$\\
\midrule
0 & $(1.1,1.3)$ & $(1.3,1.5)$\\
1 & $(0.2,0.3)$ & $(4.7,4.9)$\\
2 & $(-1.2,-1)$ & $(0.2,0.4)$\\
3 & $(-0.3,-0.2)$ & $(3.2,3.5)$\\
\bottomrule
\end{tabular}
\end{center}
These bounds give nonzero first coordinates and the stated ordering.
\end{proof}

To match the values and gradients at the cycle positions, set
\begin{equation}\label{eq:profile-data}
 \eta^i=\frac{f^i}{|z_1^i|},\qquad
 d^i=(g^i)^Tv\quad(0\le i<4),
\end{equation}
and prescribe $(h(t^i),h'(t^i))=(\eta^i,d^i)$ for $0\le i<4$.
Add $h(-1)=-1$ and $h(6)=6$. By
Proposition~\ref{prop:geometry-certificate}, the ordered list of
interpolation points and values is
\begin{equation}\label{eq:interpolation-lists}
 (-1,-1),\ (t^2,\eta^2),\ (t^0,\eta^0),\
 (t^3,\eta^3),\ (t^1,\eta^1),\ (6,6).
\end{equation}
The next proposition verifies the conditions of Lemma~\ref{lem:interpolation}
with $m=0.01$, using the enclosures in
Tables~\ref{tab:interpolation-data}--\ref{tab:interpolation-gaps}
of Appendix~\ref{app:data-bounds}.

\begin{proposition}[Interpolation certificate]\label{prop:interpolation-certificate}
The first derivatives prescribed in \eqref{eq:profile-data} satisfy
$d^i>0.01$ for $0\le i<4$.
In the list in \eqref{eq:interpolation-lists}, every two consecutive
entries $(a,A),(c,B)$ satisfy $(B-A)/(c-a)>0.01$.
\end{proposition}

\begin{proof}
Equations~\eqref{eq:cycle-gradients}, \eqref{eq:profile-coordinates}
and \eqref{eq:profile-data} give
$d^0=0.05$, $d^1=1$, $d^2=0.05p+q$ and $d^3=0.05t-p$.
The enclosures in Table~\ref{tab:interpolation-data} place
$d^2$ and $d^3$ above $0.02$.
Table~\ref{tab:interpolation-gaps} bounds all five consecutive secant
slopes below by $0.016$. Both bounds exceed $0.01$.
\end{proof}

For the line-search tests, define
\begin{equation}\label{eq:uniform-margins}
 \mathsf A^i=f^i+0.03(g^i)^Ts^i-f^{i+1},\qquad
 \mathsf W^i=(g^{i+1})^Ts^i-0.6(g^i)^Ts^i
 \quad(0\le i<4).
\end{equation}
Lemma~\ref{lem:cycle} gives $H^i\succ0$ and $(s^i)^Ty^i>0$.
Thus $s^i\ne0$ and $g^i\ne0$ by \eqref{eq:cycle-recursion}, so
$(g^i)^Ts^i=-(g^i)^TH^ig^i<0$.
Consequently $\mathsf A^i$ bounds the unit-step Armijo margin from below when
$c_1\le0.03$, and $\mathsf W^i$ bounds the Wolfe margin from below when $c_2\ge0.6$.
The next proposition verifies that both margins are positive.

\begin{proposition}[Line-search margins]\label{prop:line-search-certificate}
The margins in \eqref{eq:uniform-margins} satisfy
$\mathsf A^i>0$ and $\mathsf W^i>0$ for $0\le i<4$.
\end{proposition}

\begin{proof}
The definitions \eqref{eq:cycle-gradients} and \eqref{eq:cycle-positions},
together with $J^2=I$, give $f^{i+2}=\rho f^i$ for $i=0,1,2$.
Using \eqref{eq:cycle-gradients} and \eqref{eq:cycle-step-transformations}
in \eqref{eq:uniform-margins} therefore yields
\[
 \mathsf A^{i+2}=\rho\mathsf A^i,\qquad
 \mathsf W^{i+2}=\rho\mathsf W^i\qquad(i=0,1).
\]
Since $\rho>0$, it suffices to check the first two phases.

For the Armijo margins, substitute $f^i=(g^i)^Tx^i$ and the data in
\eqref{eq:initial-metrics}, \eqref{eq:involution-data} and
\eqref{eq:cycle-gradients}--\eqref{eq:cycle-recursion} into
\eqref{eq:uniform-margins}.
Simplifying with $P(b)=0$ from Lemma~\ref{lem:root-certificate} gives
\begin{align*}
 \mathsf A^0&=-\frac{158760b^3+163620b^2-172579b-103680}{27200b},\\
 \mathsf A^1&=\frac{79380b^3+58957b^2-70380b-46080}{13600b}.
\end{align*}
Rational interval evaluation on $[0.9178,0.9179]$, which contains $b$
by Lemma~\ref{lem:root-certificate}, gives
\begin{equation}\label{eq:armijo-enclosures}
 0.0575<\mathsf A^0<0.0611,\qquad 0.0281<\mathsf A^1<0.0312.
\end{equation}

For the Wolfe margins, \eqref{eq:initial-metrics},
\eqref{eq:cycle-gradients} and \eqref{eq:cycle-recursion} give
$(g^0)^Ts^0=-1$ and $(g^1)^Ts^1=-b$.
Combining these identities with \eqref{eq:curvatures} gives
\[
 \mathsf W^0=\frac12-0.4=0.1>0,\qquad
 \mathsf W^1=\frac{81b}{32}-0.4b=\frac{341b}{160}>0.
\]
\end{proof}

Lemma~\ref{lem:interpolation} allows arbitrary second derivatives.
The next lemma uses this freedom to realize the exact cycle with
adjustable Hessians. Their coefficients will be chosen in the next
subsection so that convergence persists under perturbations of the
initialization.

\begin{lemma}[Realization and accepted steps]\label{lem:realization}
For arbitrary $\beta^0,\ldots,\beta^3\in\R$, set
\begin{equation}\label{eq:jets}
 q^i=(-x_2^i,x_1^i)^T,\qquad
 K^i=\frac{\beta^i}{\norm{x^i}^2}q^i(q^i)^T\quad(0\le i<4).
\end{equation}
Then there is a globally Lipschitz, positively homogeneous, semialgebraic
function $\widetilde f:\R^2\to\R$ with $0\notin\Cl\widetilde f(0)$.
It is $C^2$ on $\R^2\setminus\{0\}$ and satisfies
\[
 \widetilde f(x^i)=f^i,\qquad \nabla\widetilde f(x^i)=g^i,\qquad
 \nabla^2\widetilde f(x^i)=K^i\quad(0\le i<4).
\]
For every parameter pair in \eqref{eq:parameter-range},
Algorithm~\ref{alg:bfgs} applied to $\widetilde f$ from $(x^0,H^0)$
accepts the initial trial $\alpha=1$ at every iteration, and
\[
 x^{4j+i}=r^jx^i,\qquad H^{4j+i}=r^jH^i
 \quad(j\ge0,\ 0\le i<4).
\]
\end{lemma}

\begin{proof}
\proofstep{1}{Construction of the objective}
The vectors $q^i$ in \eqref{eq:jets} are perpendicular to $x^i$, so
$K^ix^i=0$. Complete the interpolation data by prescribing
\begin{equation}\label{eq:second-derivative-data}
 h''(t^i)=k^i,
 \qquad k^i=|z_1^i|v^TK^iv\quad(0\le i<4).
\end{equation}
Apply Lemma~\ref{lem:interpolation} to the list in
\eqref{eq:interpolation-lists}, using the bounds in
Proposition~\ref{prop:interpolation-certificate}, with $m=0.01$, first
derivative one and second derivative zero at the anchors $-1,6$, and the data
\eqref{eq:profile-data} and \eqref{eq:second-derivative-data} at the other
points. Since the lemma allows arbitrary second derivatives, it applies
for every choice of the coefficients $\beta^i$.
The values at the extreme anchors are $-1$ and $6$, so the affine-tail formulas give
$h(t)=t$ for $t\notin(-2,7)$.
Define $F$ by \eqref{eq:homogeneous} and set $\widetilde f(x)=F(M^{-1}x)$.
These functions are positively homogeneous and semialgebraic.
For $z_1\ne0$, put $\sigma=\operatorname{sign}(z_1)$ and
$t=z_2/|z_1|$. Differentiation gives
\begin{equation}\label{eq:perspective-jets}
 \nabla F(z)=\binom{\sigma(h(t)-th'(t))}{h'(t)},
 \qquad
 \nabla^2F(z)=\frac{h''(t)}{|z_1|}
 \begin{pmatrix}t^2&-\sigma t\\-\sigma t&1\end{pmatrix}.
\end{equation}
The identity tails just established give $F(z)=z_2$
in a neighborhood of every nonzero point of the axis $z_1=0$.
Thus $F$ and $\widetilde f$ are $C^2$ away from zero.

\proofstep{2}{Matching the prescribed values and derivatives}
For $0\le i<4$, the prescriptions in \eqref{eq:profile-data} give
$\widetilde f(x^i)=F(z^i)=f^i$ and $\partial_2F(z^i)=d^i$.
Proposition~\ref{prop:homogeneity} gives $(z^i)^T\nabla F(z^i)=f^i$.
By \eqref{eq:profile-coordinates}, \eqref{eq:transformed-cycle-positions}
and \eqref{eq:profile-data},
$x^i=z_1^iMe_1+z_2^iMe_2$ and $d^i=(g^i)^TMe_2$.
Since $f^i=(g^i)^Tx^i$ and $z_1^i\ne0$ by
Proposition~\ref{prop:geometry-certificate}, the first component is
\[
 \partial_1F(z^i)
 =\frac{f^i-z_2^id^i}{z_1^i}
 =\frac{(g^i)^T(x^i-z_2^iMe_2)}{z_1^i}
 =(g^i)^TMe_1.
\]
Together with $\partial_2F(z^i)=d^i=(g^i)^TMe_2$, this gives
$\nabla F(z^i)=M^Tg^i$.
The chain rule then yields
$\nabla\widetilde f(x^i)=M^{-T}\nabla F(z^i)=g^i$.

For the Hessian, the symmetric matrix $M^TK^iM$ satisfies $(M^TK^iM)z^i=0$,
because $Mz^i=x^i$ and $K^ix^i=0$ by \eqref{eq:jets}.
At $z=z^i$, the notation in \eqref{eq:perspective-jets} gives
$z^i=|z_1^i|(\sigma,t)^T$ and $\sigma^2=1$, so multiplying the displayed Hessian
by $z^i$ gives $\nabla^2F(z^i)z^i=0$.
Since $z_1^i\ne0$ by Proposition~\ref{prop:geometry-certificate},
the two rows of $Bz^i=0$ for a symmetric matrix $B$ give
\[
 B_{12}=-\frac{z_2^i}{z_1^i}B_{22},\qquad
 B_{11}=\left(\frac{z_2^i}{z_1^i}\right)^2B_{22}.
\]
Thus $B$ is determined by $B_{22}$.
The $(2,2)$ entries of $M^TK^iM$ and $\nabla^2F(z^i)$ both equal
$v^TK^iv=k^i/|z_1^i|$ by \eqref{eq:profile-coordinates},
\eqref{eq:second-derivative-data} and \eqref{eq:perspective-jets}.
Hence $\nabla^2F(z^i)=M^TK^iM$.
The chain rule for $\widetilde f(x)=F(M^{-1}x)$ therefore gives
$\nabla^2\widetilde f(x^i)=K^i$.

\proofstep{3}{Lipschitz continuity and nonstationarity}
Continuity on the unit circle and homogeneity give
$\widetilde f(x)=\norm{x}\widetilde f(x/\norm{x})\to0=\widetilde f(0)$
as $x\to0$. Since $\widetilde f$ is $C^2$ away from zero,
the gradient scaling identity in \eqref{eq:homogeneity-identities} gives
\[
 \norm{\nabla\widetilde f(x)}
 \le L:=\max_{\norm{u}=1}\norm{\nabla\widetilde f(u)}<\infty
 \qquad(x\ne0).
\]
Integrating along any segment avoiding zero gives
$|\widetilde f(x)-\widetilde f(y)|\le L\norm{x-y}$.
If the segment contains zero, split it there and pass to the limit
using continuity. The two lengths add to $\norm{x-y}$, so the same
inequality holds, proving global Lipschitz continuity.

By Lemma~\ref{lem:interpolation} and \eqref{eq:perspective-jets},
$\partial_2F(z)\ge0.01$ for every $z\ne0$,
including on the nonzero axis, where $F(z)=z_2$ locally.
With $v=Me_2$ from \eqref{eq:profile-coordinates}, the chain rule gives
$\nabla\widetilde f(x)^Tv=\partial_2F(M^{-1}x)\ge0.01$
at every nonzero point.
All limiting gradients at zero, and their closed convex hull, satisfy
the same inequality. Thus $0\notin\Cl\widetilde f(0)$.

\proofstep{4}{Line-search decisions}
By Proposition~\ref{prop:homogeneity}, $x^4=rx^0$ from
Lemma~\ref{lem:cycle}, and $g^4=g^0$ from \eqref{eq:cycle-gradients}, we have
$\widetilde f(x^4)=f^4$ and $\nabla\widetilde f(x^4)=g^4=g^0$.
Proposition~\ref{prop:line-search-certificate} gives $\mathsf A^i,\mathsf W^i>0$
for the margins in \eqref{eq:uniform-margins}.
As noted after that definition, these are lower bounds for the unit-step
Armijo and Wolfe margins throughout \eqref{eq:parameter-range}.
Since $p^i=s^i$ and $x^i+s^i=x^{i+1}$, every cycle state accepts
its initial unit trial.
For $\tau>0$, Lemma~\ref{lem:bfgs-scaling} shows that a step from
$(\tau x^i,\tau H^i)$ makes the same line-search decisions and ends
at $(\tau x^{i+1},\tau H^{i+1})$.
Since $(x^4,H^4)=(rx^0,rH^0)$ by Lemma~\ref{lem:cycle}, induction
over successive four-step blocks gives the asserted formulas.
\end{proof}

\subsection{Stability of the shrinking cycle}\label{sec:stability}
Lemma~\ref{lem:realization} proves convergence to the origin for the exact
cycle and its positive rescalings. To obtain an open set of initial
states, we must show that this behavior persists under perturbations of
both $x$ and $H$. We divide the state after $4j$ steps by $r^j$,
the decay factor of the exact cycle, and show that these normalized
states remain bounded.
A related use of homogeneity and attracting cycles in normalized
coordinates appears in a construction for stochastic subgradient methods
\cite[Propositions~3.1 and~4.1]{Li2026}.

Since Lemma~\ref{lem:realization} gives an accepted step size of one along
the cycle, we study the corresponding fixed-step map. The final proof
verifies acceptance throughout an invariant neighborhood.
Represent a state $(x,H)\in\R^2\times\SPD$ by
$\omega=(x_1,x_2,H_{11},H_{12},H_{22})\in\R^5$
and write $\omega_*=(x_1^{0},x_2^{0},1,1/2,b/4+1/8)$ for the base state
from \eqref{eq:initial-metrics} and \eqref{eq:cycle-positions}.
For any fixed choice of $\beta^0,\ldots,\beta^3$ and the corresponding
function $\widetilde f$ constructed in Lemma~\ref{lem:realization}, define
the unit-step BFGS map by
\begin{equation}\label{eq:unit-step-map}
 \begin{aligned}
  s&=-H\nabla\widetilde f(x),\qquad
  y=\nabla\widetilde f(x+s)-\nabla\widetilde f(x),\qquad
  \widehat H=\mathcal B(H,s,y),\\
  \mathcal T(\omega)&=(x_1+s_1,x_2+s_2,
    \widehat H_{11},\widehat H_{12},\widehat H_{22}),
 \end{aligned}
\end{equation}
where $\mathcal B$ is given by \eqref{eq:update}.
Its domain is the open set
\begin{equation}\label{eq:unit-step-domain}
 \mathcal D=\{\omega\in\R^5:H\succ0,\ x\ne0,\ x+s\ne0,\ s^Ty>0\}.
\end{equation}
Write $\mathcal T^j$ for the $j$-fold composition wherever defined,
with $\mathcal T^0=\mathrm{id}$. The next lemma records the smoothness and
scaling properties needed to study the shrinking cycle.

\begin{lemma}[Local BFGS dynamics]\label{lem:local-dynamics}
The map $\mathcal T$ is $C^1$ on $\mathcal D$.
For every $\omega\in\mathcal D$ and every $\tau>0$,
\[
 \mathcal T(\tau\omega)=\tau\mathcal T(\omega),\qquad
 \mathcal T^4(\omega_*)=r\omega_*.
\]
\end{lemma}

\begin{proof}
On $\mathcal D$ in \eqref{eq:unit-step-domain},
$\widetilde f$ is $C^2$ near $x$ and $x+s$,
and $s^Ty>0$ keeps the denominators in \eqref{eq:update}
nonzero. Hence \eqref{eq:unit-step-map} defines a $C^1$ map there.
By the gradient scaling in Proposition~\ref{prop:homogeneity},
replacing $\omega$ by $\tau\omega$ multiplies $s$ by $\tau$ and leaves
$y$ unchanged. This preserves $H\succ0$, the conditions $x\ne0$ and
$x+s\ne0$, and the inequality $s^Ty>0$.
Hence $\tau\omega\in\mathcal D$ for every $\tau>0$.
The update scaling in Lemma~\ref{lem:bfgs-scaling} and
\eqref{eq:unit-step-map} then give
$\mathcal T(\tau\omega)=\tau\mathcal T(\omega)$.
Applying $\mathcal T$ along the exact cycle in Lemma~\ref{lem:realization} gives
$\mathcal T^4(\omega_*)=r\omega_*$.
\end{proof}

For a sequence $\omega^{k+1}=\mathcal T(\omega^k)$, normalize each
four-step block by setting $\xi^j=r^{-j}\omega^{4j}$.
Write $\mathcal D_4=\operatorname{dom}(\mathcal T^4)$, an open set
invariant under positive scaling by Lemma~\ref{lem:local-dynamics}.
On this set, define the normalized return map
\begin{equation}\label{eq:fixed-return-map}
 \mathcal F(\omega)=r^{-1}\mathcal T^4(\omega).
\end{equation}
The same lemma gives, whenever the updates are defined,
\begin{equation}\label{eq:block-recursion}
 \xi^{j+1}=\mathcal F(\xi^j),\qquad
 \omega^{4j+i}=r^j\mathcal T^i(\xi^j)
 \quad(j\ge0,\ 0\le i<4).
\end{equation}
Thus, if the normalized states $\xi^j$ form a bounded
sequence, then $\omega^{4j}=r^j\xi^j\to0$, since
$r=64/81\in(0,1)$ by the construction in Section~\ref{sec:cycle}.
The map $\mathcal F$ is $C^1$ and positively homogeneous on $\mathcal D_4$,
and $\mathcal F(a\omega_*)=a\omega_*$ for every $a>0$.
We will show that, for $\xi^0$ sufficiently close to $\omega_*$, the
sequence $(\xi^j)$ converges to a point of the fixed ray
$\{a\omega_*:a>0\}$ and is therefore bounded.

The Hessian coefficients in \eqref{eq:jets} control the derivative of
$\mathcal F$ without changing the exact cycle or its line-search decisions,
by Lemma~\ref{lem:realization}. To obtain attraction, choose
\begin{equation}\label{eq:hessian-coefficients}
 (\beta^0,\beta^1,\beta^2,\beta^3)=(1.9,1.6,-4.1,0.2).
\end{equation}
Fix the corresponding objective $\widetilde f$ and maps
$\mathcal T,\mathcal F$. The next lemma proves attraction to the fixed
ray: deviations from the ray contract, while changes along it are summable.

\begin{lemma}[Attraction to the fixed ray]\label{lem:attraction}
There are arbitrarily small open neighborhoods $U\subset\mathcal D_4$
of $\omega_*$ such that $\mathcal F(U)\subset U$ and, for every $\xi^0\in U$, the iterates
$\xi^{j+1}=\mathcal F(\xi^j)$ converge to a positive multiple of $\omega_*$.
\end{lemma}

\begin{proof}
\proofstep{1}{Integral identities}
Since $(\omega_*)_3=1$, write each state uniquely as
$\xi=\xi_3\omega_*+u$, where
$u\in W:=\{u\in\R^5:u_3=0\}$.
For $\xi'=\mathcal F(\xi)$, set $u'=\xi'-\xi'_3\omega_*$.
Suppose $\xi_3>0$ and the segment $\xi(t)=\xi_3\omega_*+tu$,
$0\le t\le1$, lies in $\mathcal D_4$.
Since $\mathcal F$ is $C^1$ and
$\mathcal F(\xi_3\omega_*)=\xi_3\omega_*$, the fundamental theorem of
calculus applied to its third component gives
\[
 \xi'_3-\xi_3
 =(\mathcal F(\xi))_3-(\mathcal F(\xi_3\omega_*))_3
 =\int_0^1(D\mathcal F(\xi(t))u)_3\,dt.
\]
Using this identity and applying the same theorem to each component of
$\mathcal F$ gives
\begin{align*}
 u'&=\xi'-\xi'_3\omega_*\\*
   &=\bigl(\mathcal F(\xi)-\mathcal F(\xi_3\omega_*)\bigr)
     -(\xi'_3-\xi_3)\omega_*\\*
   &=\int_0^1D\mathcal F(\xi(t))u\,dt
     -\left(\int_0^1(D\mathcal F(\xi(t))u)_3\,dt\right)\omega_*\\*
   &=\int_0^1\bigl(D\mathcal F(\xi(t))u
     -(D\mathcal F(\xi(t))u)_3\omega_*\bigr)\,dt.
\end{align*}
The vector in the last integrand belongs to $W$, since its third
coordinate is zero. For any norm $\norm{\cdot}_\perp$ on $W$, the
triangle inequalities for these integrals yield
\begin{equation}\label{eq:transverse-integral-bounds}
\begin{aligned}
 |\xi'_3-\xi_3|
 &\le\int_0^1 |(D\mathcal F(\xi(t))u)_3|\,dt,\\
 \norm{u'}_\perp
 &\le\int_0^1
 \norm{D\mathcal F(\xi(t))u-(D\mathcal F(\xi(t))u)_3\omega_*}_\perp\,dt.
\end{aligned}
\end{equation}

\proofstep{2}{Operator-norm bounds}
For each fixed $z\in\mathcal D_4$, define the linear maps
$L_z:W\to W$ and $\ell_z:W\to\R$ by
\begin{equation}\label{eq:transverse-operators}
 L_z(v)=D\mathcal F(z)v-(D\mathcal F(z)v)_3\omega_*,\qquad
 \ell_z(v)=(D\mathcal F(z)v)_3.
\end{equation}
The two integrands in \eqref{eq:transverse-integral-bounds} are therefore
$|\ell_{\xi(t)}(u)|$ and $\norm{L_{\xi(t)}(u)}_\perp$.
We choose a norm on $W$ in which $L_{\omega_*}$ is a contraction,
then use continuity to bound both operators uniformly near $\omega_*$.
Define $R:\R^5\to\R^4$ by $R(v)=(v_1,v_2,v_4,v_5)$.
Since $R|_W:W\to\R^4$ is an isomorphism, there is a unique matrix
$A\in\R^{4\times4}$ satisfying
\begin{equation}\label{eq:transverse-derivative}
 R\bigl(L_{\omega_*}(v)\bigr)=AR(v)\qquad(v\in W).
\end{equation}
For a vector $y$, write $\norm{y}_\infty=\max_i|y_i|$.
For a matrix $B=(b_{ij})$, the induced matrix norm is
\begin{equation}\label{eq:matrix-infinity-norm}
 \norm{B}_\infty
 =\sup_{y\ne0}\frac{\norm{By}_\infty}{\norm{y}_\infty}
 =\max_i\sum_j|b_{ij}|.
\end{equation}
Proposition~\ref{prop:direct-stability}, proved in
Section~\ref{app:stability-inequalities}, supplies an invertible matrix
$S$ with $\norm{SAS^{-1}}_\infty\le0.78$.
Equip $W$ with the norm $\norm{v}_\perp=\norm{SR(v)}_\infty$.
We use operator norms induced by this norm on $W$ and the absolute
value on $\R$. By \eqref{eq:transverse-derivative} and the isomorphism
$SR:W\to\R^4$,
\[
 \norm{L_{\omega_*}}_{\mathrm{op}}
 =\sup_{v\in W\setminus\{0\}}
   \frac{\norm{SAR(v)}_\infty}{\norm{SR(v)}_\infty}
 =\norm{SAS^{-1}}_\infty\le0.78<1.
\]
Choose $\gamma\in(0.78,1)$ and
$C>\norm{\ell_{\omega_*}}_{\mathrm{op}}$.
Because $D\mathcal F$ is continuous, both $L_z$ and $\ell_z$ depend
continuously on $z$ in these operator norms.
Choose $\varepsilon\in(0,1)$ and $\delta>0$ small enough that every
state $z$ satisfying $|z_3-1|\le\varepsilon$ and
$\norm{z-z_3\omega_*}_\perp\le\delta$ belongs to $\mathcal D_4$ and obeys
\[
 \norm{L_z}_{\mathrm{op}}\le\gamma,\qquad
 \norm{\ell_z}_{\mathrm{op}}\le C.
\]
If $|\xi_3-1|\le\varepsilon$ and $\norm{u}_\perp\le\delta$, then
$(\xi(t))_3=\xi_3$ and
$\norm{\xi(t)-\xi_3\omega_*}_\perp=t\norm{u}_\perp\le\delta$.
Thus every point of the segment satisfies these operator bounds, and
\begin{align*}
 |\ell_{\xi(t)}(u)|
 &\le\norm{\ell_{\xi(t)}}_{\mathrm{op}}\norm{u}_\perp
 \le C\norm{u}_\perp,\\*
 \norm{L_{\xi(t)}(u)}_\perp
 &\le\norm{L_{\xi(t)}}_{\mathrm{op}}\norm{u}_\perp
 \le\gamma\norm{u}_\perp
 \qquad(0\le t\le1).
\end{align*}
Substituting into \eqref{eq:transverse-integral-bounds} gives
\begin{equation}\label{eq:transverse-bounds}
 \norm{u'}_\perp\le\gamma\norm{u}_\perp,\qquad
 |\xi'_3-\xi_3|\le C\norm{u}_\perp
\end{equation}
throughout this range of $\xi_3$ and $u$.

\proofstep{3}{An invariant neighborhood}
Choose $M\ge\max\{C/(1-\gamma),\varepsilon/\delta\}$ and set
\[
 U=\{\xi\in\R^5:|\xi_3-1|+M\norm{\xi-\xi_3\omega_*}_\perp<\varepsilon\},
\]
which is an open neighborhood of
$\omega_*$ in $\R^5$, and $M\ge\varepsilon/\delta$ gives
$U\subset\mathcal D_4$.
Decreasing $\varepsilon$ and $\delta$ makes $U$ arbitrarily small.
By \eqref{eq:transverse-bounds},
\[
 |\xi'_3-1|+M\norm{u'}_\perp
 \le |\xi_3-1|+(C+M\gamma)\norm{u}_\perp
 \le |\xi_3-1|+M\norm{u}_\perp<\varepsilon,
\]
so $\mathcal F(U)\subset U$.
For $\xi^0\in U$, write $\xi^j=\xi_3^j\omega_*+u^j$.
Iterating \eqref{eq:transverse-bounds} gives
$\norm{u^j}_\perp\le\gamma^j\norm{u^0}_\perp$ and
\[
 \sum_{j=0}^{\infty}|\xi_3^{j+1}-\xi_3^j|
 \le \frac{C\norm{u^0}_\perp}{1-\gamma}<\infty.
\]
Hence $u^j\to0$, and $\xi_3^j$ converges. Since $\xi^j\in U$ for
every $j$, the definition of $U$ gives $\xi_3^j>1-\varepsilon$.
Its limit is therefore at least $1-\varepsilon>0$.
\end{proof}

\subsection{Completion of the main proof}\label{sec:completion}
We choose an invariant neighborhood on which every four-step block satisfies
the line-search tests and stays inside the ball of radius ten.
We then define the coercive objective and prove by induction that its
BFGS iterates follow the scaled blocks.
\begin{proof}[Proof of Theorem~\ref{thm:main}]
\proofstep{1}{The initial set}
Use the objective $\widetilde f$ fixed in Section~\ref{sec:stability}
with the coefficients \eqref{eq:hessian-coefficients}.
For $\xi\in\mathcal D_4$, let $(\widehat x^i,\widehat H^i)$ be the state
represented by $\mathcal T^i(\xi)$ and set
$\widehat g^i=\nabla\widetilde f(\widehat x^i)$ for $0\le i\le4$.
For $0\le i<4$, write
$\widehat s^i=\widehat x^{i+1}-\widehat x^i=-\widehat H^i\widehat g^i$.
These quantities depend on $\xi$ and are continuous on $\mathcal D_4$
by Lemma~\ref{lem:local-dynamics}.
The domain \eqref{eq:unit-step-domain} ensures smooth evaluation points,
positive definite matrices and
$(\widehat s^i)^T(\widehat g^{i+1}-\widehat g^i)>0$.
In particular, $\widehat g^i\ne0$ and $(\widehat g^i)^T\widehat s^i<0$ for $0\le i<4$.

At $\xi=\omega_*$, Lemmas~\ref{lem:cycle} and~\ref{lem:realization}
give $\|\widehat x^i\|<10$ for $0\le i\le4$, and
Proposition~\ref{prop:line-search-certificate} gives strict Armijo and
Wolfe inequalities at $(c_1,c_2)=(0.03,0.6)$.
These finitely many strict inequalities persist on an open neighborhood
of $\omega_*$ by continuity.
Since Lemma~\ref{lem:attraction} supplies arbitrarily small invariant
neighborhoods, choose one such $U\subset\mathcal D_4$ with
$\mathcal F(U)\subset U$ so that, for every $\xi\in U$,
\begin{equation}\label{eq:initial-block-conditions}
 \begin{aligned}
  \|\widehat x^i\|&<10 &&(0\le i\le4),\\
  \widetilde f(\widehat x^{i+1})
  &<\widetilde f(\widehat x^i)+0.03(\widehat g^i)^T\widehat s^i
  &&(0\le i<4),\\
  (\widehat g^{i+1})^T\widehat s^i
  &>0.6(\widehat g^i)^T\widehat s^i &&(0\le i<4).
 \end{aligned}
\end{equation}
Because $(\widehat g^i)^T\widehat s^i<0$, these tests also hold for every
parameter pair in \eqref{eq:parameter-range}.
Define the initial set in the theorem by
\begin{equation}\label{eq:initial-set}
 \mathcal U=\{(x,H)\in\R^2\times\SPD:
 (x_1,x_2,H_{11},H_{12},H_{22})\in U\}.
\end{equation}
The state-coordinate map is a homeomorphism onto its image, so
$\mathcal U$ is open and contains the cycle state $(x^0,H^0)$
constructed in Section~\ref{sec:cycle}.

\proofstep{2}{The coercive objective}
Define
\begin{equation}\label{eq:coercive}
 f(x)=\widetilde f(x)+\max\{0,\norm{x}^2-100\}.
\end{equation}
By Lemma~\ref{lem:realization}, $\widetilde f$ has a global Lipschitz
constant $L$. Since $\widetilde f(0)=0$, we have
$\widetilde f(x)\ge-L\norm{x}$, and hence
\[
 f(x)\ge\norm{x}^2-L\norm{x}-100\longrightarrow+\infty
 \qquad\text{as }\norm{x}\to\infty.
\]
Thus $f$ is coercive. It is locally Lipschitz and semialgebraic, and
agrees with $\widetilde f$ on the ball of radius ten.
In particular, $\Cl f(0)=\Cl\widetilde f(0)$, so
Lemma~\ref{lem:realization} gives $0\notin\Cl f(0)$.

\proofstep{3}{Induction and convergence}
Fix a parameter pair in \eqref{eq:parameter-range} and an initial state
in $\mathcal U$ from \eqref{eq:initial-set}, represented by $\omega^0\in U$.
Define $\xi^0=\omega^0$ and $\xi^{j+1}=\mathcal F(\xi^j)$ using
\eqref{eq:fixed-return-map}.
Since $\mathcal F(U)\subset U$, this sequence is well defined, with $\xi^j\in U$
for every $j\ge0$.
We prove by induction that the actual algorithm applied to $f$ is defined
at every iteration, accepts every unit trial, and has states
\begin{equation}\label{eq:actual-block-states}
 \omega^{4j+i}=r^j\mathcal T^i(\xi^j)
 \qquad(j\ge0,\ 0\le i\le4).
\end{equation}

The initial identity is $\omega^0=\xi^0$.
Suppose the algorithm has reached $\omega^{4j}=r^j\xi^j$ with all
preceding unit trials accepted. Keep $j$ fixed and, for $0\le i\le4$,
let $(\widehat x^i,\widehat H^i)$ be the state
represented by $\mathcal T^i(\xi^j)$.
We show that the state identity propagates from phase $i$ to phase $i+1$.
Suppose the state at $k=4j+i$, $0\le i<4$, is
$\omega^k=r^j\mathcal T^i(\xi^j)$, so
$(x^k,H^k)=(r^j\widehat x^i,r^j\widehat H^i)$.
At $(\widehat x^i,\widehat H^i)$, the tests in
\eqref{eq:initial-block-conditions} accept the unit trial for $\widetilde f$,
and \eqref{eq:unit-step-map} gives the next state
$(\widehat x^{i+1},\widehat H^{i+1})$.
Lemma~\ref{lem:bfgs-scaling}, with $\tau=r^j$, therefore gives an accepted
unit step from $(r^j\widehat x^i,r^j\widehat H^i)$ to
$(r^j\widehat x^{i+1},r^j\widehat H^{i+1})$.

Both evaluation points are nonzero by \eqref{eq:unit-step-domain}, and
their norms are below ten by \eqref{eq:initial-block-conditions} and
$0<r^j\le1$. Thus \eqref{eq:coercive} makes $f$ and $\widetilde f$
agree locally at the current point and the unit trial.
The objective values and gradients used by Algorithm~\ref{alg:bfgs}
coincide, so it accepts the same trial and produces
$\omega^{k+1}=r^j\mathcal T^{i+1}(\xi^j)$.
By \eqref{eq:unit-step-map}--\eqref{eq:unit-step-domain}, each unscaled phase
has positive curvature and hence a nonzero gradient. These properties persist under positive
scaling, excluding termination or breakdown.

Applying this implication successively for $i=0,1,2,3$ shows that the
iterations $4j,\ldots,4j+3$ are well defined and accept their unit trials,
and that \eqref{eq:actual-block-states} holds throughout this block.
By \eqref{eq:fixed-return-map}, the state after these four iterations is
\[
 \omega^{4(j+1)}=r^j\mathcal T^4(\xi^j)
 =r^{j+1}\mathcal F(\xi^j)=r^{j+1}\xi^{j+1},
\]
which closes the induction on $j$.
Finally, combining \eqref{eq:actual-block-states} with the position bounds
in \eqref{eq:initial-block-conditions} yields
$\|x^{4j+i}\|<10r^j$ for $j\ge0$ and $0\le i<4$.
Since $0<r<1$, the iterates converge to the origin.
\end{proof}

\section{Concluding remarks}\label{sec:conclusion}
Random initialization does not rule out convergence to a
non-Clarke-stationary point for the BFGS method with Armijo-Wolfe line search.
The two-dimensional example admits a nonempty open set of such initial
states in the joint space of points and inverse-Hessian estimates,
uniformly over the line-search parameters in Theorem~\ref{thm:main}.
The objective is coercive, locally Lipschitz and semialgebraic, and
Gaussian--Wishart initialization assigns positive probability to this
open set. These properties give a negative resolution of
\cite[Challenge~7.1(2)]{LewisOverton2013}.
Positive homogeneity near the origin sustains a shrinking cycle, while
attraction extends the failure to nearby initializations.
Rigorous computational bounds certify the finite inequalities used in
the construction and its analytical stability argument.

An important remaining question is whether Algorithm~\ref{alg:bfgs}
has only Clarke-stationary cluster points almost surely for convex,
locally Lipschitz, semialgebraic objectives with bounded sublevel sets.
For nonconvex objectives with the same regularity and boundedness
assumptions, which alternative line-search rules or safeguards on the
inverse-Hessian estimates can recover this guarantee?

\section*{Acknowledgments}
The author thanks Adrian Lewis for his helpful feedback and suggestions
on the presentation of this work. The author used Codex to assist with the mathematical
construction, numerical experiments, verification code, and manuscript
preparation. GPT-6 Pro was consulted on developing and simplifying the
example and proofs. AI-assisted material was refined through repeated
discussion, mathematical and computational checks, and editing.
The author independently verified all mathematical
arguments and assumes responsibility for all content.

\appendix

\section{Rational certificates for the finite inequalities}\label{app:certificates}
This appendix provides two certificates. Appendix \ref{app:data-bounds} supplies the bounds
in Proposition~\ref{prop:interpolation-certificate}, which ensure that
Lemma~\ref{lem:interpolation} can construct $h$.
Appendix \ref{app:stability-inequalities} verifies the contraction needed for attraction to the fixed
ray in Lemma~\ref{lem:attraction}.
Verification details and reproducible computations are provided in the
accompanying GitHub repository linked on the first page of the paper.

\paragraph{Certified arithmetic.}
We use standard interval arithmetic with outward rounding;
see \cite[Sections~2.3--2.4 and~3.2]{MooreKearfottCloud2009}.
Every terminating decimal denotes an exact rational number. An interval
$[a,b]$ represents all real numbers between its rational endpoints.
Addition, multiplication and negation are evaluated by
\[
 [a,b]+[c,d]=[a+c,b+d],\qquad
 [a,b][c,d]=[\min(ac,ad,bc,bd),\max(ac,ad,bc,bd)],
\]
with $-[a,b]=[-b,-a]$. If $0\notin[a,b]$, its reciprocal is
$[a,b]^{-1}=[1/b,1/a]$; division by an interval containing zero is rejected.
With $Q=2^{256}$, round every input and computed interval outward by
\begin{equation}\label{eq:interval-rounding}
 [a,b]\longmapsto
 \left[\frac{\lfloor Qa\rfloor}{Q},\frac{\lceil Qb\rceil}{Q}\right].
\end{equation}
Printed decimal endpoints are again rounded outward.
For interval matrices, membership and arithmetic are
understood entrywise, with matrix products formed by interval sums and
products.

For both certificates, write the enclosure from
Lemma~\ref{lem:root-certificate} as
\begin{equation}\label{eq:direct-root-box}
 b\in Y=[0.91781039164, 0.91781039165].
\end{equation}

\subsection{Interpolation bounds}\label{app:data-bounds}
Consider the data constructed in Section \ref{sec:cycle}. 
Substituting \eqref{eq:involution-data} into \eqref{eq:cycle-positions}
and simplifying gives the expression used in these calculations:
\begin{equation}\label{eq:short-xzero}
 x^0=\left(
 -\frac{9(8978b^2-6527b-2304)}{5440b},\,
 \frac{9(1764b^2-2671b+1488)}{2720}\right)^T,
\end{equation}
where $b\in Y$ by \eqref{eq:direct-root-box}.
Use the positions from
\eqref{eq:short-xzero} and \eqref{eq:cycle-positions} and the gradients from
\eqref{eq:cycle-gradients}.
Using the coordinate enclosures in the proof of
Proposition~\ref{prop:geometry-certificate}, evaluate
$t^i=z_2^i/|z_1^i|$ and the data in \eqref{eq:profile-data}.
Their enclosing intervals are
given in Table~\ref{tab:interpolation-data}.
\begin{table}[H]
\setlength{\belowcaptionskip}{6pt}
\centering
\caption{Certified interpolation data at the four cycle positions.}
\label{tab:interpolation-data}
\small
\begin{tabular}{cccc}
\toprule
$i$ & $t^i$ & $\eta^i$ & $d^i$\\
\midrule
0 & [1.4026, 1.4027] & [1.0701, 1.0702] & [0.0499, 0.0501] \\
1 & [4.7721, 4.7722] & [4.7721, 4.7722] & [1.0000, 1.0000] \\
2 & [0.2790, 0.2791] & [1.0520, 1.0521] & [0.0291, 0.0292] \\
3 & [3.3299, 3.3300] & [4.3954, 4.3955] & [0.9887, 0.9888] \\
\bottomrule
\end{tabular}
\end{table}

For the list in \eqref{eq:interpolation-lists}, number the five
consecutive gaps $j=0,\ldots,4$ from left to right. For adjacent entries
$(a,A),(c,B)$, evaluate $\ell=c-a$ and $(B-A)/\ell$ by the interval
rules above.
The resulting enclosures, used in the proof of
Proposition~\ref{prop:interpolation-certificate}, are given in
Table~\ref{tab:interpolation-gaps}.
\begin{table}[H]
\setlength{\belowcaptionskip}{6pt}
\centering
\caption{Certified distances and secant slopes between consecutive interpolation points.}
\label{tab:interpolation-gaps}
\begin{tabular}{ccc}
\toprule
$j$ & $\ell$ & $(B-A)/\ell$\\
\midrule
0 & [1.2790, 1.2791] & [1.6043, 1.6044] \\
1 & [1.1236, 1.1237] & [0.0161, 0.0162] \\
2 & [1.9272, 1.9273] & [1.7253, 1.7254] \\
3 & [1.4421, 1.4422] & [0.2612, 0.2613] \\
4 & [1.2278, 1.2279] & [0.9999, 1.0001] \\
\bottomrule
\end{tabular}
\end{table}

\subsection{Contraction toward the fixed ray}\label{app:stability-inequalities}
For the norm used in the proof of Lemma~\ref{lem:attraction}, choose
\begin{equation}\label{eq:stability-coordinates}
 S=\begin{pmatrix}
 464&1077&4501&-1708\\
 1366&2793&10000&552\\
 -448&-1058&-4343&1648\\
 16&293&1787&-2075
 \end{pmatrix}.
\end{equation}
The next proposition certifies contraction in this norm.

\begin{proposition}\label{prop:direct-stability}
For the matrix $A$ in \eqref{eq:transverse-derivative} and the matrix $S$
in \eqref{eq:stability-coordinates}, $S$ is invertible and
$\|SAS^{-1}\|_\infty\le0.78$.
\end{proposition}

\begin{proof}
We first compute $A$ by differentiating at the exact cycle, then enclose
the resulting rational expression by interval arithmetic over $Y$.

\proofstep{1}{Differentiating one BFGS step}
Recall the cycle data $g^i,H^i,x^i$ from \eqref{eq:initial-metrics}
and \eqref{eq:cycle-gradients}--\eqref{eq:cycle-positions}, the steps $s^i$
and gradient differences $y^i$ from \eqref{eq:cycle-recursion}, and
the Hessians $K^i$ from \eqref{eq:jets}, with the coefficients
\eqref{eq:hessian-coefficients}. Set $K^4=K^0/r$.
Lemma~\ref{lem:realization} and the Hessian scaling in
Proposition~\ref{prop:homogeneity} give
$K^i=\nabla^2\widetilde f(x^i)$ for $0\le i\le4$.

In the state coordinates of \eqref{eq:unit-step-map}, write
$\omega^i=(x_1^i,x_2^i,H_{11}^i,H_{12}^i,H_{22}^i)$ for $0\le i\le4$.
Thus $\omega^0=\omega_*$ and $\mathcal T(\omega^i)=\omega^{i+1}$
for $0\le i<4$ by Lemmas~\ref{lem:cycle} and~\ref{lem:realization}.
For a state variation $\delta\omega\in\R^5$, set
\[
 \delta x=\bigl((\delta\omega)_1,(\delta\omega)_2\bigr)^T,\qquad
 \delta H=\begin{pmatrix}
 (\delta\omega)_3&(\delta\omega)_4\\
 (\delta\omega)_4&(\delta\omega)_5
 \end{pmatrix}.
\]
Fix $0\le i<4$. The symbol $\delta$ denotes differentiation at
$\varepsilon=0$ along $\omega=\omega^i+\varepsilon\delta\omega$.
Write $\widehat\omega=\mathcal T(\omega)$ for the updated state,
whose point and matrix are $\widehat x=x+s$ and
$\widehat H=\mathcal B(H,s,y)$, and set $g=\nabla\widetilde f(x)$.
Since $x^i+s^i=x^{i+1}$, differentiation gives
\begin{equation}\label{eq:state-variation}
 \begin{aligned}
 \delta g&=K^i\delta x,&
 \delta s&=-((\delta H)g^i+H^i\delta g),\\
 \delta\widehat x&=\delta x+\delta s,&
 \delta y&=K^{i+1}\delta\widehat x-\delta g.
 \end{aligned}
\end{equation}
All quantities in the following formulas are evaluated at this cycle
state, so $H=H^i$, $s=s^i$ and $y=y^i$. Introduce the scalars $c,d$,
the vector $t$ and the matrices $N,V$ by
\[
 c=s^Ty,\qquad t=Hy,\qquad d=y^THy,\qquad
 N=ts^T+st^T,\qquad V=ss^T.
\]
Expansion of \eqref{eq:update} gives
\begin{equation}\label{eq:certificate-bfgs-expansion}
 \widehat H=H-\frac{N}{c}+\frac{c+d}{c^2}V.
\end{equation}
The product and quotient rules give
\begin{equation}\label{eq:bfgs-differential}
 \begin{aligned}
 \delta c&=(\delta s)^Ty+s^T\delta y,\qquad
 \delta t=(\delta H)y+H\delta y,\\
 \delta d&=2(\delta y)^Tt+y^T(\delta H)y,\qquad
 \delta V=(\delta s)s^T+s(\delta s)^T,\\
 \delta N&=(\delta t)s^T+t(\delta s)^T
              +(\delta s)t^T+s(\delta t)^T,\\
 \delta\widehat H
 &=\delta H-\frac{\delta N}{c}+\frac{N\delta c}{c^2}
   +\frac{c+d}{c^2}\delta V+\left(\frac{\delta c+\delta d}{c^2}
          -\frac{2(c+d)\delta c}{c^3}\right)V.
 \end{aligned}
\end{equation}
\proofstep{2}{The four-step derivative}
Set $D^i=D\mathcal T(\omega^i)\in\R^{5\times5}$.
Equations~\eqref{eq:state-variation} and \eqref{eq:bfgs-differential} give
\begin{equation}\label{eq:phase-jacobian}
 \delta\widehat\omega=D^i\delta\omega
 =\bigl((\delta\widehat x)_1,(\delta\widehat x)_2,
 (\delta\widehat H)_{11},(\delta\widehat H)_{12},
 (\delta\widehat H)_{22}\bigr)^T.
\end{equation}
Its $j$th column is obtained by taking $\delta\omega$ to be the $j$th
standard basis vector of $\R^5$.
Since $\omega^0=\omega_*$ and $\mathcal F=r^{-1}\mathcal T^4$ by
\eqref{eq:fixed-return-map}, the chain rule gives
\begin{equation}\label{eq:four-step-jacobian}
 D_{\mathrm{tot}}:=D^3D^2D^1D^0
 =D\mathcal T^4(\omega_*),\qquad
 D\mathcal F(\omega_*)=r^{-1}D_{\mathrm{tot}}.
\end{equation}
Recall $W=\{v\in\R^5:v_3=0\}$, $R(v)=(v_1,v_2,v_4,v_5)$ and the
map $L_{\omega_*}$ defined in \eqref{eq:transverse-operators}.
By \eqref{eq:four-step-jacobian},
\[
 L_{\omega_*}(v)
 =r^{-1}\bigl(D_{\mathrm{tot}}v-(D_{\mathrm{tot}}v)_3\omega_*\bigr)
 \qquad(v\in W).
\]
Using \eqref{eq:transverse-derivative} and writing
$(\kappa_1,\kappa_2,\kappa_3,\kappa_4)=(1,2,4,5)$ for the coordinates
retained by $R$, we obtain
\begin{equation}\label{eq:transverse-extension}
 A_{ij}
 =\frac{(D_{\mathrm{tot}})_{\kappa_i\kappa_j}
       -(\omega_*)_{\kappa_i}(D_{\mathrm{tot}})_{3\kappa_j}}{r}
 \qquad(1\le i,j\le4).
\end{equation}
\proofstep{3}{Entry bounds and row sums}
The formulas above express every entry of $A$ rationally in $b$.
To enclose them, replace $b$ by an independent variable $\lambda\in Y$
in the formulas defining $g^i,H^i,x^i,K^i,s^i,y^i$, using
\eqref{eq:short-xzero} for $x^0$.
An argument $\lambda$ denotes this substitution, for example
$x^i(\lambda)$ and $s^i(\lambda)$.
Evaluate \eqref{eq:state-variation}, \eqref{eq:bfgs-differential} and
\eqref{eq:phase-jacobian} with these data to obtain $D^i(\lambda)$.
Form $D_{\mathrm{tot}}(\lambda)=D^3(\lambda)D^2(\lambda)D^1(\lambda)D^0(\lambda)$
and set
\[
 \omega^0(\lambda)=\bigl(x_1^0(\lambda),x_2^0(\lambda),
 1,1/2,\lambda/4+1/8\bigr).
\]
In the right-hand side of \eqref{eq:transverse-extension}, replace
$D_{\mathrm{tot}}$ and $\omega_*$ by $D_{\mathrm{tot}}(\lambda)$ and
$\omega^0(\lambda)$ to define $\widehat A(\lambda)$.
Since $D_{\mathrm{tot}}(b)=D_{\mathrm{tot}}$ and $\omega^0(b)=\omega_*$,
we have $\widehat A(b)=A$. The derivative interpretation established
above applies at $\lambda=b$.

Apart from fixed nonzero constants, the denominator factors in
\eqref{eq:involution-data}, \eqref{eq:cycle-positions} and
\eqref{eq:jets} are $\lambda$, $2\lambda-1$ and
$\norm{x^i(\lambda)}^2$. Formula~\eqref{eq:bfgs-differential} also
divides by powers of $s^i(\lambda)^Ty^i(\lambda)$.
For every $\lambda\in Y$, we have $\lambda>0.9$ and
$2\lambda-1>0.8$, and interval evaluation gives
\[
 \norm{x^i(\lambda)}^2>0.7,\qquad
 s^i(\lambda)^Ty^i(\lambda)>0.4\qquad(0\le i<4).
\]
The remaining factors $r=64/81$ and $1-r=17/81$ are fixed and positive.

Since $\det S=-22\,230\,591\,050\ne0$, $S$ is invertible.
Evaluate $S\widehat A(\lambda)S^{-1}$ over $Y$ by interval arithmetic
and denote the resulting $4\times4$ interval matrix by $\mathbf M$.
Every denominator interval excludes zero, so the inclusion
principle~\cite[equation~(5.17)]{Rump2010} ensures that each interval
$\mathbf M_{ij}$ contains $(S\widehat A(\lambda)S^{-1})_{ij}$ for
every $\lambda\in Y$.

For an interval $[a,c]$, set $\operatorname{mag}([a,c])=\max\{|a|,|c|\}$.
Rounding $\operatorname{mag}(\mathbf M_{ij})$ upward to three decimal
places gives the nonnegative matrix $\Gamma$:
\begin{equation}\label{eq:stability-entry-magnitudes}
 \begin{gathered}
 |(S\widehat A(\lambda)S^{-1})_{ij}|
 \le\operatorname{mag}(\mathbf M_{ij})\le\Gamma_{ij}\qquad(\lambda\in Y),\\
 \Gamma=\begin{pmatrix}
 0.133 & 0.164 & 0.049 & 0.011\\
 0.240 & 0.091 & 0.094 & 0.015\\
 0.042 & 0.006 & 0.521 & 0.207\\
 0.024 & 0.002 & 0.211 & 0.478
 \end{pmatrix}.
 \end{gathered}
\end{equation}
Summing each row of this displayed matrix gives
\begin{equation}\label{eq:direct-stability-bound}
 \left(\sum_{j=1}^4|(S\widehat A(\lambda)S^{-1})_{ij}|\right)_{i=1}^4
 \le(0.357,\,0.440,\,0.776,\,0.715)\qquad(\lambda\in Y),
\end{equation}
where the inequality is componentwise. Since $\widehat A(b)=A$, the
row-sum formula \eqref{eq:matrix-infinity-norm} gives
$\|SAS^{-1}\|_\infty\le0.776<0.78$.
\end{proof}

\bibliographystyle{alpha}
\bibliography{references}
\end{document}